\documentclass{amsart}

\usepackage{caption}
\usepackage{graphicx}
\usepackage[T1]{fontenc}
\usepackage{verbatim}
\usepackage[margin=1.3in]{geometry}
\usepackage{hyperref}
\usepackage{graphics}
\usepackage{color}

\newtheorem{theorem}{Theorem}[section]
\newtheorem{lemma}[theorem]{Lemma}
\newtheorem{prop}[theorem]{Proposition}

\theoremstyle{definition}
\newtheorem{definition}[theorem]{Definition}

\newtheorem{corollary}[theorem]{Corollary}
\newtheorem{question}[theorem]{Question}
\newtheorem{proposition}[theorem]{Proposition}

\newtheorem{remark}[theorem]{Remark}

\numberwithin{equation}{section}

\newcommand{\trace}{\text{Tr}}

\usepackage[usenames,dvipsnames,svgnames,table]{xcolor}
\begin{document}

\title[Left-Orderability, Branched Covers and Double Twist Knots]{Left-Orderability, Branched Covers\\and Double Twist Knots}

\bibliographystyle{alpha}

\author{Hannah Turner}
\address{Department of Mathematics, University of Texas at Austin}
%\curraddr{Department of Mathematics and Statistics,Case Western Reserve University, Cleveland, Ohio 43403}
\email{hannahturner@math.utexas.edu}
%\thanks{The author was supported by an NSF graduate research fellowship}

  % Information for second author
%\author{Author Two}
%\address{Mathematical Research Section, School of Mathematical Sciences,
%Australian National University, Canberra ACT 2601, Australia}

%    General info
%\subjclass[2000]{Primary 54C40, 14E20; Secondary 46E25, 20C20}

%\date{January 1, 2001 and, in revised form, June 22, 2001.}


%\keywords{Differential geometry, algebraic geometry}
\begin{abstract}
For some families of two-bridge knots, including double-twist knots with genus at least four, we determine precisely the set of integers $n>1$ such that the fundamental group of the $n$-fold cyclic branched cover of the 3-sphere along these knots is left-orderable. There are knots, including the figure-eight knot, for which this set is empty. We give the first class of hyperbolic knots, not of this type, for which these integers can be completely determined.
\end{abstract}

\maketitle

\section{Introduction}

%% The correct journal style for \specialsection is all uppercase; a known bug
%% in amsart.cls prevents this, so input must be uppercase until it is fixed.
%\specialsection*{This is a Special Section Head}
%\specialsection*{THIS IS A SPECIAL SECTION HEAD}
%This is an example of a special section head%
%%%%%%%%%%%%%%%%%%%%%%%%%%%%%%%%%%%%%%%%%%%%%%%%%%%%%%%%%%%%%%%%%%%%%%%%
%\footnote{Here is an example of a footnote. Notice that this footnote
%text is running on so that it can stand as an example of how a footnote
%with separate paragraphs should be written.
%\par
%And here is the beginning of the second paragraph.}%
%%%%%%%%%%%%%%%%%%%%%%%%%%%%%%%%%%%%%%%%%%%%%%%%%%%%%%%%%%%%%%%%%%%%%%%%

A non-trivial group $G$ is called \textit{left-orderable} if it admits a strict total ordering $(G,<)$ that is left-invariant, i.e. whenever $g<h$ for $g,h\in G$ we also have $fg<fh$ for any $f\in G$. Examples of left-orderable groups coming from topology and dynamics abound. Braid groups \cite{DDRW08}, the group of orientation preserving homeomorphisms of the line $\mbox{Homeo}^+(\mathbb{R})$, and the fundamental group of a connected, compact, orientable 3-manifold with positive first Betti number are all known to be left-orderable \cite{BRW05}. 

However, there are both examples and non-examples of left-orderable groups among fundamental groups of irreducible 3-manifolds which are rational homology spheres. We will call a 3-manifold \textit{left-orderable} if its fundamental group is left-orderable.  According to the L-space conjecture, such a 3-manifold $M$ should be left-orderable if and only if $M$ is \textit{not} a Heegaard Floer L-space if and only if $M$ admits a co-oriented taut foliation \cite{BGW13, Ju15}.

Let $\Sigma_n(L)$ denote the $n$-fold cyclic branched cover of the 3-sphere branched over a link $L$. Such manifolds have provided a rich family on which to investigate the L-space conjecture. The conjecture has been confirmed for $\Sigma_2(L)$ when $L$ is a non-split alternating link \cite{OS05, BGW13} or the closure of a 3-braid \cite{Bal08, BH19, LW14}.

Combined work of Hu  and Gordon shows for a two-bridge knot $K$ with non-zero signature that $\Sigma_n(K)$ is left-orderable for $n$ sufficiently large \cite{Hu15, Go17}. On the other hand there are examples of two-bridge knots all of whose cyclic branched covers are L-spaces \cite{Pe09} and whose fundamental groups are not left-orderable for any index $n\geq 2$ \cite{DPT05}.

\begin{center}
\vspace{8pt}
\includegraphics[scale=.8]{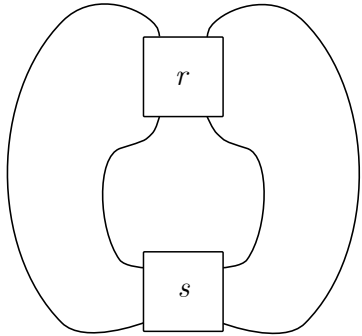} 
\captionof{figure}{The knot $J(r,s)$ where $r$ and $s$ count the number of signed half twists in each box.}
\label{DoubleT}
\end{center}

In this paper we study the left-orderability of $\Sigma_n(K)$ for a class of two-bridge knots. The family of double twist knots is a two-parameter family $J(r,s)$ as in Figure \ref{DoubleT}. For $r$ and $s$ even, we have the following theorems of D\c{a}bkowski-Przytycki-Togha and Tran.

\begin{theorem}[Theorem 2(c) in \cite{DPT05}]\label{DPT thm}
Let $k$ and $m$ be positive integers and $K=J(2k,-2m)$. Then $\Sigma_n(K)$ is not left-orderable for any $n$.
\end{theorem}

\begin{theorem}[Theorem 1 in \cite{Tr15}]\label{Tran thm}
Let $k$ and $m$ be positive integers and $K=J(2k,2m)$. Then $\Sigma_n(K)$ is left-orderable for $n\geq \pi/\cos^{-1}\sqrt{1-(4km)^{-1}}$.
\end{theorem}
 
Tran also gave bounds on the index $n$ for which $\Sigma_n(K)$ is left-orderable for $K=J(2k+1,2m)$. These bounds grow with $k$ and $m$ \cite[Theorem 2]{Tr15}. 

We note if $r$ and $s$ are even, then the three-genus is $g(J(r,s))=1$. On the other hand, any double-twist knot with $g(J(r,s))>1$ can be written $J(2k+1,2m)$ where $g(J(2k+1,2m))=|m|$, see Section \ref{notation}. Improving on Tran's result, for double-twist knots with genus at least four, we completely classify the indices $n$ for which $\Sigma_n(K)$ is left-orderable. If the genus is two or three, we decide these indices with one or two exceptions. 

\begin{theorem}\label{Bigtheorem} Let $k\geq 1$ be fixed. Then $\Sigma_n(J(2k+1,2m))$ is left-orderable in the following cases:

\begin{enumerate}
\item $n\geq 3$ when $m\leq -3$
\item $n\geq 4$ when $m=-2$
\item $n\geq 5$ when $m=2$
\item $n\geq 4$ when $m=3$
\item $n\geq 3$ when $m\geq 4$.
\end{enumerate}

\noindent In cases $(1)$, $(2)$ and $(5)$, $\Sigma_n(J(2k+1,2m))$ is left-orderable if and only if the index $n$ satisfies the corresponding inequality.
\end{theorem}

\begin{remark} Now the only cyclic branched covers of double twist knots with genus at least two for which left-orderability remains unknown are $\Sigma_4(J(2k+1,4))$  and $\Sigma_3(J(2k+1,6))$. For the proof of this statement see the discussion at the end of the section. 
\end{remark}

Theorem \ref{Bigtheorem} answers a question of Boileau, Boyer and Gordon in the affirmative.
They show $\Sigma_n(J(2k+1,2m))$ is not an L-space for index $n\geq 4$ when $m=-2$ and for index $n\geq 3$ when $m\leq -3$, and ask if these manifolds are left-orderable for the same indices \cite[Problem 12.11]{BBG19}. The L-space conjecture predicts that $\Sigma_n(K)$ should also admit a taut foliation for these same indices. In the case where $K=J(2k+1,2m)$ with $m$ odd  and negative, Gordon and Lidman showed that $\Sigma_n(K)$ is an integer homology sphere, and admits a taut foliation when $n$ divides $m$ \cite{GL14}. 

%We point out that any L-space conjecture statemtnkdkkj

The techniques in Theorem \ref{Bigtheorem} can be extended to families of two-bridge knots which are not double-twist knots. The simplest extension might be to two-bridge knots with three twist-box regions. While in many cases our techniques could prove left-orderability even for low-index branched covers of these knots, our method of proof became more involved as the twist parameters grew. Nevertheless, we include one such generalization of Theorem \ref{Bigtheorem}.

\begin{center}
%\begin{figure}[h]
\vspace{8pt}
\includegraphics[scale=.25]{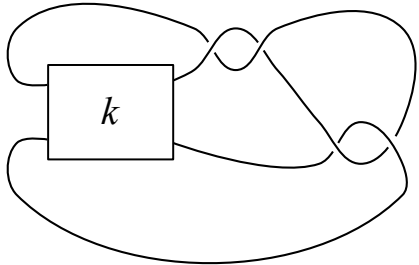} 
\captionof{figure}{The knot $K_l$ where $k=2l-1$ counts the signed half twists in the box.}
\label{notDT}
%\end{figure}
\end{center}

Let $K_l$ denote the knot in Figure \ref{notDT}. We assume that $l\geq 2$ in which case, it can be shown that $K_l$ is \textit{not} a double-twist knot, and $g(K_l)=l$.

\begin{theorem}\label{bigtwistythm}
 $\Sigma_n(K_l)$ is left-orderable in the following cases:

\begin{enumerate}
\item $n\geq 5$ when $l=2$
\item $n\geq 4$ when $l=3$
\item $n\geq 3$ when $l\geq 4$.
\end{enumerate}
 
 \noindent For case $(3)$, $\Sigma_n(K_l)$ is left-orderable if and only if the index $n\geq 3$.
 
\end{theorem}

Teragaito showed that there are two-bridge knots, with arbitrarily large genus, all of whose cyclic-branched covers are L-spaces by exhibiting them as double-branched covers of alternating knots \cite{Te14}. All of the cyclic branched covers of these knots are also not left-orderable \cite[Theorem 8]{BGW13}. On the other hand, Theorems \ref{Bigtheorem} and \ref{bigtwistythm} give evidence that there is a relationship between the properties of the L-space conjecture for $\Sigma_n(K)$ for a knot $K$ and its three-genus if $\Sigma_n(K)$ is left-orderable for some $n$. We point out that results relating genus and the L-space conjecture in branched cyclic covers have been achieved by Ba in the following theorems.

\begin{theorem}[Theorem 1.3 in \cite{Ba19}]\label{Ba} Let $K$ be a two-bridge knot with $g(K)=2$. Then $\Sigma_3(K)$ is an L-space, and is not left-orderable.
\end{theorem}

\begin{theorem}[Corollary 1.4 in \cite{Ba19}] Let $K$ be a two-bridge knot with $g(K)=1$. Then $\Sigma_n(K)$ is an L-space, and is not left-orderable for $n\leq 5$.
\end{theorem}

Two-bridge knots have lens space branched double covers. These manifolds have finite fundamental groups and hence are not left-orderable. Together with Theorem \ref{Ba}, this allows us to conclude that the bounds we obtain in Theorem \ref{Bigtheorem} are best possible in the cases that $m\leq -2$ or $m\geq 4$, and for Theorem \ref{bigtwistythm} they are best possible for $l\geq 4$.

\subsection{$\widetilde{PSL}(2,\mathbb{R})$-representations of 3-manifold groups}

The method used here to prove left-orderability of 3-manifold groups, is to construct non-trivial $\widetilde{PSL}(2,\mathbb{R})$ representations; this technique has a long history, see eg \cite{EHN81,Hu15,Tr15, CD18}. The $n$-fold cyclic branched covers of hyperbolic knots in the 3-sphere are all hyperbolic for $n\geq 3$ with the exception of the figure-eight knot \cite[Corollary 1.26]{CHK}. The knots we consider are hyperbolic; hence all of the manifolds $\Sigma_n(K)$ in Theorems \ref{Bigtheorem} and \ref{bigtwistythm} are hyperbolic. Thus the manifold groups we consider always have a $PSL(2,\mathbb{C})$ representation, the holonomy representation. 

It is another question altogether whether these 3-manifolds admit a $\widetilde{PSL}(2,\mathbb{R})$ represention or even one into $PSL(2,\mathbb{R})$. Gao described an infinite family of hyperbolic non-L-spaces, which by the L-space conjecture are expected to be left-orderable, with no non-trivial $PSL(2,\mathbb{R})$ representations \cite{Ga17}. It is perhaps surprising then that the manifolds in Theorem \ref{Bigtheorem} are not only all left-orderable but also have $\widetilde{PSL}(2,\mathbb{R})$ representations.

\subsection{The Set of Left-Orderable Indices}

The inspiration to prove Theorem \ref{Bigtheorem} came from a desire to understand the form the following set can take:
\[\mathcal{LO}_{br}(K)=\{n\geq 2: \Sigma_n(K) \mbox{ is left-orderable}\}\]

\noindent for a fixed knot. Boileau, Boyer, and Gordon studied the set 
\[\mathcal{L}_{br}(K)=\{n\geq 2:\Sigma_n(K) \mbox{ is an L-space}\}\]

\noindent for strongly-quasipositive knots \cite{BBG19}. Note that if the L-space conjecture holds for cyclic branched covers of knots in the 3-sphere, then $ \mathcal{LO}_{br}(K)\sqcup \mathcal{L}_{br}(K)=\{n: n\geq 2\}$.

This notation echos that used for the interval of L-space slopes $\mathcal{L}(Y)$ for $Y$ a compact, oriented three-manifold with boundary \cite{RR17}. Evidence suggests, as for $\mathcal{L}(Y)$, there are constraints on the form the sets $\mathcal{L}_{br}(K)$ and $\mathcal{LO}_{br}(K)$ can take. For all known examples $\mathcal{L}_{br}(K)$ is either $\emptyset$, $\{n:2\leq n\leq N\}$ for some $N\geq 2$, or $\{n:n\geq 2\}$. In particular, for strongly quasipositive fibered knots, $\mathcal{L}_{br}(K)\subset \{n: 2\leq n\leq 5\}$ \cite{BBG19}. There are no knots for which $N$ is known to be larger than $5$.

Similarly from the left-orderability perspective, for known examples, $\mathcal{LO}_{br}(K)$ is either $\{n: n\geq N\}$ for some $N\geq 2$, or $\emptyset$. There are no knots for which $N$ is known to be greater than 6.

\begin{question} Is it true that $\mathcal{LO}_{br}(K)$ is either empty or the set $\{n: n\geq N\}$ for some integer $N\geq 2$? Can $N$ be arbitrarily large?
\end{question}

While Theorems \ref{Bigtheorem} and \ref{bigtwistythm} do not completely determine $\mathcal{LO}_{br}(K)$ for $K=J(r,s)$ or $K=K_l$ respectively, they do allow us to give a characterization the behavior of $\mathcal{LO}_{br}(K)$.

\begin{corollary} Let $K=J(r,s)$ with $g(K)\geq 2$, or $K=K_l$ for $l\geq 2$. Then $\mathcal{LO}_{br}(K)$ always takes the form $\{n: n\geq N\}$ with $N\leq 5$.
\end{corollary}

\subsection{Outline}
We begin by describing the relationship between the left-orderability of $\Sigma_n(K)$ and roots of a Riley polynomial of $K$ in Section \ref{non-ab-rep}. In Section \ref{Cheb} we include background on Chebyshev polynomials. We give a formula for a Riley polynomial of double-twist knots in Section \ref{Riley}. In Section \ref{roots} we prove Theorem \ref{Bigtheorem} by finding the desired roots of the Riley polynomials. We conclude with the proof of Theorem \ref{bigtwistythm} in Section \ref{moreresults}.

\subsection{Conventions}\label{notation}

Two-bridge knots are the closures of rational tangles, and so have an (non-unique) associated fraction $p/q$ with $-p<q<p$, see \cite{Ka96, BZ03}.We will write $K(p,q)$ to denote the unique two-bridge knot associated to the fraction $p/q$. 

We now note some facts about the family $J(r,s)$:

\begin{enumerate}
\item If $rs$ is odd, then $J(r,s)$ is a link of two components.
\item If $rs=0$ then $K$ is the unknot. 
\item $J(-r,-s)\cong -J(r,s)$.
\item $J(r,s)\cong J(s,r)$.
\end{enumerate}

Excluding the cases of the unknot and links of multiple components, we can consider without loss of generality knots of the form $J(r,2m)$, with $|r|,|m|>0$. The manifolds $\Sigma_n(K)\cong -\Sigma_n(-K)$ are orientation-reversing homeomorphic; we are interested in the fundamental groups of these manifolds so we need only consider one of $K$ or $-K$. Hence we can further assume that $r>0$.

We present no new results in the case that $g(J(r,s))=1$, so we exclude the case that both parameters $r$ and $s$ are even. Thus, we consider double-twist knots of the form $K=J(2k+1,2m)$ with $|m|\geq 2$. Finally, if $k=0$ then $K$ is a $(2,2p+1)$-torus knot for some integer $p$. Gordon and Lidman completely determined the indices for which the branched covers of these knots are left-orderable \cite{GL14, GL17}. In summary, when $K$ is a double-twist knot, we will assume $K=J(2k+1,2m)$ with $|m|\geq 2$ and $k\geq 1$.

\subsection{Acknowledgments}
The author is grateful to her advisor Cameron Gordon for his time, expertise and interest in this project. The author also thanks Jonathan Johnson for many helpful conversations. The author is supported by an NSF graduate research fellowship under grant no. DGE-1610403.

\section{Non-abelian Representations for Two-bridge Knots}\label{non-ab-rep}

This section follows work of Hu \cite{Hu15} to relate left-orderability of branched covers of two-bridge knots to finding roots of certain polynomials. 

A fundamental result of Boyer, Rolfsen and Wiest allows one to prove left-orderability of a (compact, orientable, irreducible) 3-manifold group by instead finding a non-trivial representation into a group known to be left-orderable \cite{BRW05}. The fact that $\widetilde{ SL(2,\mathbb{R})}$ is left-orderable \cite{Ber91} has been exploited to prove that certain 3-manifold groups are left-orderable, including in the proof of the following result.

\begin{theorem}[Theorem 3.1 of \cite{Hu15}]\label{HuLO}Let $K$ be a prime knot in $S^3$ and $X_K$ denote its complement. Let $Z$ be a meridional element of $\pi_1(X_K)$. If there exists a non-abelian representation $\rho:\pi_1(X_K)\to  SL(2,\mathbb{R})$ such that $\rho(Z^n)=\pm I$ then  $\Sigma_n(K)$ is left-orderable.
\end{theorem}

Let $K$ be a two-bridge knot for the remainder of the section.  Then the knot group has a presentation of the form
\begin{equation}\label{generalpres}
\pi_1(X_K)=\langle a,b:va=bv\rangle
\end{equation}

\noindent where $a$ and $b$ are meridians and $v$ is a word in $a$ and $b$, see eg. \cite{Ka96}.

A non-abelian representation  $\rho:\pi_1(X_K)\to SL(2,\mathbb{C})$ can be conjugated to be of the form:

\begin{align}\label{nonabrep}
\rho(a)=A=\begin{bmatrix}
s&1\\
0&s^{-1}
\end{bmatrix} & &\rho(b)=B=\begin{bmatrix}
s&0\\
2-y & s^{-1}
\end{bmatrix}
\end{align}

where $s\in \mathbb{C}\setminus \{0\}$ and $y\in\mathbb{C}$ satisfying $VA-BV=0$ where $V=\rho(v)$. A special case of (\ref{nonabrep}) is

\begin{align}\label{goodrep}
\rho(a)=A=\begin{bmatrix}
e^{i\pi/n}&1\\
0&e^{-i\pi/n}
\end{bmatrix} & &\rho(b)=B=\begin{bmatrix}
e^{i\pi/n}&0\\
2-y & e^{-i\pi/n}
\end{bmatrix}
\end{align}

\noindent for fixed $n\in\mathbb{N}$ with $n\geq 2$. This map is closer to satisfying the conditions of Theorem \ref{HuLO} since it can be shown that $\rho(a^n)=-I$.

The map defined by (\ref{nonabrep}) can be defined for any presentation of the form defined in (\ref{generalpres}), though it is not necessarily a homomorphism. To check that the map is in fact a representation for a given $(s,y)\in\mathbb{C}^*\times\mathbb{C}$, we need to see that $VA-BV=0$ is satisfied. In Section \ref{Riley} we compute the entries of $R=VA-BV$ explicitly in the case $K=J(2k+1,2m)$ as in \cite{Tr15}. 

For two-bridge knots, work of Riley shows that determining when the map (\ref{nonabrep}) is a representation reduces to determining when exactly one entry of the matrix $R$ is zero. 

\begin{proposition}[Theorem 1 in \cite{Ril84}]\label{onepoly}$R_{i,j}=0$ for $1\leq i,j\leq 2$ if and only if $R_{1,2}=0$. In other words, if $R_{1,2}=0$ then the map in (\ref{nonabrep}) is a homomorphism.
\end{proposition}
\noindent We see that $R_{1,2}=R_{1,2}(s,y)$ can be considered as a polynomial in $\mathbb{Z}[s^{\pm 1},y]$.

\begin{proposition}[Proposition 1 in \cite{Ril84}] We have that $R_{1,2}(s,y)=R_{1,2}(s^{-1},y)$. Thus $R_{1,2}(s,y)=f(s+s^{-1},y)$ where $f$ is a two-variable polynomial with coefficients in $\mathbb{Z}$.
\end{proposition}
\begin{definition} Let $K$ be a two-bridge knot, and fix a presentation for $\pi_1(X_K)$. Let $x=s+s^{-1}$. Then we will call $\phi_K(x,y):=f(s+s^{-1},y)$ a Riley polynomial of $K$. 
\end{definition}

We note that the polynomial $\phi_K(x,y)$ is not an invariant of $K$, but depends on the choice of presentation for $\pi_1(X_K)$. The following statement should be compared to Hu's Proposition 4.1 and the proof of Theorem 4.3 \cite{Hu15}.

\begin{theorem}\label{itsallaboutRiley}
Let $K$ be a two-bridge knot, and let $\phi_K(x,y)$ be a Riley polynomial of $K$. Fix $n\geq 2$. Suppose there exists $y_n>2$ a real solution of $\phi_K(2\cos(\pi/n),y)$. Then $\Sigma_n(K)$ is left-orderable.
\end{theorem}

\begin{proof}
Since $\phi_K(2\cos(\pi/n),y_n)=0$ it is clear that $R_{1,2}(e^{\pi i/n},y_n)=0$.
Thus, setting $y=y_n$ in (\ref{goodrep}) defines a $ SL(2,\mathbb{C})$ representation of $\pi_1(X_K)$ by Proposition \ref{onepoly}. In addition, $y_n>2$ is real, so a result of Khoi tells us that (\ref{goodrep}) can be conjugated to a representation $\rho'$ into $ SL(2,\mathbb{R})$  \cite[p. 786]{Kh03}. Since $\rho(a^n)=-I$ we also have that $\rho'(a^n)=-I$. Finally, two-bridge knots are prime; we can now see that $\rho'$ satisfies the conditions of Theorem \ref{HuLO}, and we conclude that $\Sigma_n(K)$ is left orderable for that particular $n$. 
\end{proof}

\noindent The following theorem therefore implies Theorem \ref{Bigtheorem}, and will be proved in Section \ref{roots}.

\begin{samepage}
\begingroup
\def\thetheorem{\ref{Polybigtheorem}}
\begin{theorem}
Fix $n\geq 2$, and let $K=J(2k+1,2m)$. Then there is a presentation of $\pi_1(X_K)$ with Riley polynomial $\phi_K(x,y)$ such that $\phi_K(2\cos(\pi/n),y)$ has a root $y_n>2$ in the following cases:

\begin{enumerate}
\item $n\geq 3$ when $m\leq -3$
\item $n\geq 4$ when $m=-2$
\item $n\geq 5$ when $m=2$
\item $n\geq 4$ when $m=3$
\item $n\geq 3$ when $m\geq 4$.
\end{enumerate}
\end{theorem}
\addtocounter{theorem}{-1}
\endgroup
\end{samepage}

\section{Chebyshev Polynomials}\label{Cheb}

Let $S_n(z)$ be the sequence of Chebyshev polynomials defined by the recurrence relation $S_{n+1}(z)=zS_{n}(z)-S_{n-1}(z)$ with $S_0(z)=1$ and $S_1(z)=z$. They allow simplifications of certain recurrences. For a well-chosen presentation of $\pi_1(X_K)$ for $K$ a double-twist knot, the Riley polynomial $\phi_K(x,y)$ can expressed in terms of these polynomials, and their properties allow us to understand the roots of $\phi_K(x,y)$.

\begin{lemma}[Lemma 3.2 in \cite{TrPretz}]\label{bettersequence} If $a_n$ is a sequence of complex numbers satisfying $a_{n+1}=ca_n-a_{n-1}$ for some $c\in\mathbb{C}$, then $a_{n+1}=S_{n}(c)a_1-S_{n-1}(c)a_0$.
\end{lemma}

\begin{remark}
Calling them Chebyshev polynomials is apt since  $S_n(2z)=U_n(z)$ where $U_n(z)$ are the Chebyshev polynomials of the second kind defined by $U_0(z)=1$, $U_1(z)=2z$ and $U_n(z)=2zU_{n-1}(z)-U_{n-2}(z)$.
\end{remark}

We will make use of properties of these Chebyshev polynomials in many arguments. One can allow $n$ to be negative and extend the recurrence; we do not need this generalization, so we will assume that $n\geq 0$ for the remainder of the section.

\begin{lemma}\label{Srel} The Chebyshev polynomials $S_n(z)$ satisfy the following:
\begin{enumerate}
\item{$S_n(2)=n+1$ and $S_n(-2)=(-1)^{n}(n+1)$}
\item{The roots of $S_n(z)$ are $2\cos\left(\frac{k\pi}{n+1}\right)$ for $k=1,2,\ldots,n$.}
\item{ $S_{n}(t)> 0$ when $t\geq 2$ for $t\in\mathbb{R}$.}
\item{The inequality $S_{n+1}(t)> S_{n}(t)$ holds when $t\geq 2$ for $t\in\mathbb{R}$.}
\end{enumerate}
\end{lemma}

\begin{proof}\
\begin{enumerate}

\item{This follows easily by induction.
}

\item{Using the fact that the roots of $U_n(x)$ are $\cos\left(\frac{k\pi}{n+1}\right)$ for $k=1,2,\ldots,n$, the result follows from the fact that $S_n(2z)=U_n(z)$.}

\item{It is clear from the definition of Chebyshev polynomials that the leading coefficient is positive so that the end behavior as $t$ tends to infinity is positive.  By (2) we have that all of the roots lie in the interval $(-2,2)$. Thus, $S_N(t)$ is positive on $[2,\infty)$. }

\item{We proceed by induction. For $n=0$ or $1$ the statement is clear. Now suppose that the statement holds for all $0\leq n<N$ and let $N>1$.

Let $t\geq 2$. We have by the induction hypothesis that $S_{N}(t)> S_{N-1}(t)$. Hence,
\begin{align*}
S_{N+1}(t)&=tS_{N}(t)-S_{N-1}(t)\geq 2S_{N}(t)-S_{N-1}(t)\\
&= S_{N}(t)+(S_{N}(t)-S_{N-1}(t))> S_{N}(t).\qedhere
\end{align*} }

\end{enumerate}
\end{proof}

\begin{lemma}\label{Sbehav} Let $n\geq 1$. Then $(-1)^nS_n(t)<0$ on the interval $\left( 2\cos\left(\frac{m\pi}{m+1}\right), 2\cos\left(\frac{(m-1)\pi}{m+1}\right)\right)$.
\end{lemma}

\begin{proof}
We begin by noting that $r_1=2\cos\left(\frac{m\pi}{m+1}\right)$ is the smallest root of $S_n(t)$, and that $r_1$ and $r_2=2\cos\left(\frac{(m-1)\pi}{m+1}\right)$ are consecutive roots. Thus the sign of $S_n(t)$ on $(r_1,r_2)$ is constant and opposite of the sign on $(-\infty,r_1)$ which is also constant. Since $-2\in (-\infty, r_1)$ and  $S_n(-2)=(-1)^n(n+1)$, the lemma follows.
\end{proof}

\section{A formula for the Riley Polynomial}\label{Riley}

Let $K=J(2k+1,2m)$. We will fix a presentation for $\pi_1(X_K)$. For Sections \ref{Riley} and \ref{roots} when we write $\phi_K(x,y)$ we mean the Riley polynomial of $K$ for the following choice of presentation:
\begin{align}\label{favoritepresentation}
\pi_1(X_K)=\langle a,b \,|\, w^ma=bw^m\rangle
\end{align}

\noindent where $a$ and $b$ are meridians and  $w=(ba^{-1})^kba(b^{-1}a)^k$ \cite{HS04}.

An easy consequence of Lemma \ref{bettersequence} gives a formula for powers of matrices in $SL(2,\mathbb{C})$ in terms of Chebyshev polynomials. 

\begin{lemma}[Lemma 2.2 of \cite{MT14}]\label{matrixpower}
Let $M\in  SL(2,\mathbb{C})$. Then
 \[M^{n}=S_{n}(\trace (M))I-S_{n-1}(\trace (M))M^{-1}\!\!\!\!\!\!\!.\]
\end{lemma}

Let $\rho$ be the map in (\ref{nonabrep}), and let $A=\rho(a)$, $B=\rho(b)$ and $W=\rho(w)$. Recall that $x=s+s^{-1}=\trace(A)=\trace(B)$ and note that $\trace(BA^{-1})=\trace(B^{-1}A)=y$.

\begin{lemma}[Lemma 2.3 in \cite{MT14}]\label{almostrep}
\begin{align*}
WA-BW=\begin{bmatrix}
0& \alpha\\
(y-2)\alpha&0
\end{bmatrix}
\end{align*}

where $\alpha=\alpha_k(x,y)=1+(y+2-x^2)S_{k-1}(y)(S_k(y)-S_{k-1}(y))$.
\end{lemma}

\begin{lemma}[Lemma 2.4 in \cite{MT14}]\label{trace}The trace of $W$ is given by 
\[\lambda=\lambda_k(x,y)=\trace(W)=x^2-y-(y-2)(y+2-x^2)S_k(y)S_{k-1}(y).\]
\end{lemma}

\noindent We prove a mild reformulation of a proposition of Morifuji-Tran \cite[Proposition 2.5]{MT14}.

\begin{prop}\label{rep?} If $m\geq 1$ then $\phi_K(x,y)= S_{m-1}(\lambda)\alpha-S_{m-2}(\lambda)$. If $m\leq -1$ then $\phi_K(x,y)=S_{|m|}(\lambda)-S_{|m|-1}(\lambda)\alpha$.
\end{prop}

\begin{proof}
Let $m\geq 1$. Since Lemma \ref{trace} computes the trace of $W$, Lemma \ref{matrixpower} allows us to simplify $W^m$. Lemma \ref{almostrep} allows us to simplify further and conclude the following series of equalities.
\begin{align*}
R_{i,j}=W^mA-BW^m&=S_{m-1}(\lambda)WA-S_{m-2}(\lambda)A-S_{m-1}(\lambda)BW+S_{m-2}(\lambda)B\\
&=S_{m-1}(\lambda)\left(WA-BW\right)-S_{m-2}(\lambda)\left(A-B\right)\\
&=S_{m-1}(\lambda)\begin{bmatrix}
0&\alpha\\
(y-2)\alpha&0
\end{bmatrix}-S_{m-2}(\lambda)\begin{bmatrix}
0&1\\
(y-2)&0
\end{bmatrix}\\
&=\begin{bmatrix}
0&S_{m-1}(\lambda)\alpha-S_{m-2}(\lambda)\\
(y-2)(S_{m-1}(\lambda)\alpha-S_{m-2}(\lambda))&0
\end{bmatrix}
\end{align*}
\vspace{5pt}
\noindent Now let $m\leq -1$ and note that $\trace(W^{-1})=\trace(W)$. We have:
\begin{align*}
R_{i,j}=W^mA-BW^m&=\left(W^{-1}\right)^{|m|}A-B\left(W^{-1}\right)^{|m|}\\
&=S_{|m|}(\lambda)A-S_{|m|-1}(\lambda)WA-S_{|m|}(\lambda)B+S_{|m|-1}(\lambda)BW\\
&=S_{|m|}(\lambda)(A-B)-S_{|m|-1}(\lambda)(WA-BW)\\
&=S_{|m|}(\lambda)\begin{bmatrix}
0&1\\
(y-2)&0
\end{bmatrix}-S_{|m|-1}(\lambda)\begin{bmatrix}
0&\alpha\\
(y-2)\alpha&0
\end{bmatrix}\\
&=\begin{bmatrix}
0& S_{|m|}(\lambda)-S_{|m|-1}(\lambda)\alpha\\
(y-2)(S_{|m|}(\lambda)-S_{|m|-1}(\lambda)\alpha)&0
\end{bmatrix}\qedhere
\end{align*}

\end{proof}

\section{Roots of the Riley Polynomial and Double-Twist Knots}\label{roots}

Because of Theorem \ref{itsallaboutRiley}, the next section is devoted to finding roots larger than $2$ of $\phi_K(x,y)$. We now assume $y\in\mathbb{R}$. Some results of the section hold for any $x\in \mathbb{R}$; some only follow, or follow more easily in the case we take $x=x_n=e^{\pi i/n}+e^{-\pi i/n}=2\cos(\pi/n)$. Our applications of the lemmas of the subsequent sections only require the statements in the case that $x_n=2\cos(\pi/n)$.

\begin{lemma}\label{leading} For fixed $x\in\mathbb{R}$, we have that:
\[\lim_{y\to\infty}\phi_K(x,y)=\left\{\begin{array}{cc}
\infty & \text{if } m \text{ odd and positive or } m \text{ even and negative}\\
-\infty &\text{if } m \text{ even and positive or } m \text{ odd and negative}.
\end{array}\right.\]
\end{lemma}

\begin{proof}
Let $m\geq 1$ and $l,a,s_m$ denote the leading term of $\lambda(x,y),\alpha(x,y)$ and $S_{m}(y)$ respectively as functions of $y$.  Then the leading term $p$ of $\phi(x,y)=S_{m-1}(\lambda)\alpha-S_{m-2}(\lambda)$ as a function of $y$ is $p=s_{m-1}(l)a$. It is not hard to see that $l=-y^2s_ks_{k-1}$, $a=ys_ks_{k-1}$ and $s_m=y^m$. Thus the sign of $p$ depends only on the parity of $m$. In particular, the degree of $p$ is positive when $m$ is odd and negative when $m$ is even. A similar argument gives that the leading term of $\phi_K(x,y)$ is $-s_{|m|-1}(l)a$ when $m\leq -1$.
\end{proof}

Our goal in this section is to prove Theorem \ref{Polybigtheorem}; the proof of the case $m=2$ differs slightly from the general case. We prove this case first.

\begin{proposition}\label{m=2} Let $K=J(2k+1,4)$. Then $\phi_K(x_n,y)$ has a root $y_n>2$ for $n\geq 5$.
\end{proposition}
\begin{proof}
Here $m=2$, so we have that $\phi_K(x,y)=S_1(\lambda)\alpha-S_0(\lambda)=\lambda(x,y)\alpha(x,y)-1$ by Proposition \ref{rep?}. By Lemma \ref{leading} we know that there is $y_-> 2$ so that $\phi_K(x,y_-)<0$ for any real $x$. Considering the Riley polynomial at $y=2$ we see that:
\[\phi_K(x,2)=(x^2-2)(1+(4-x^2)k)-1\geq (x^2-2)(1+(4-x^2))-1> (x^2-2)-1=x^2-3\]

\noindent so long as $k\geq 1$. For $n\geq 6$ we have $x_n^2=(2\cos(\pi/n))^2\geq (2\cos(\pi/6))^2=3$. Hence $\phi_K
(x_n,2)> 0$ for $n\geq 6$. By the intermediate value theorem, there must be a root $y_n> 2$. 

A direct computation for $n=5$ shows we can do slightly better. Computing:
 \begin{align*}
 \phi_K(x_5,2)&=(2\cos(\pi/5)^2-2)(1+(4-2\cos(\pi/5)^2)k)-1\\
 &=\left(\frac{-1+\sqrt{5}}{2}\right)\left(\frac{7-\sqrt{5}}{2}\right)k-1\\
 &=\left(\frac{8\sqrt{5}-12}{4}\right)k-1> k-1\geq 0
\end{align*}

\noindent gives that the Riley polynomial is positive for $y_+=2$. Again we get a root $y_5>2$ by the intermediate value theorem.
\end{proof}

\begin{lemma}\label{alpha}
Let $x_n=2\cos({\pi/n})$. For $y\geq 2$, we have that $\alpha=\alpha(x_n,y)> 1$ for all $n\geq 2$.
\end{lemma}

\begin{proof}
Recall that $\alpha(x_n,y)=1+(y+2-x_ n^2)S_{k-1}(y)\left(S_k(y)-S_{k-1}(y)\right)$. Lemma \ref{Srel}(3) gives that $S_{k-1}(y)> 0$ for all $y\geq 2$, and Lemma \ref{Srel}(4) gives that $S_k(y)-S_{k-1}(y)> 0$ for $y\geq 2$. Finally we have that $-2<x_n< 2$ and in particular $x_n^2< 4$ so that $(y+2-x_n^2)>0$ for $y\geq 2$. Thus, $\alpha(x_n,y)-1=(y+2-x_n^2)S_{k-1}(y)\left(S_k(y)-S_{k-1}(y)\right)> 0$ for $y\geq 2$ as it is a product of positive functions.
\end{proof}

\begin{lemma}\label{anylambda} Fix $x\in\mathbb{R}$. For any $c\leq x^2-2$ there exists $y_c\geq 2$ such that $\lambda(x,y_c)=c$.
\end{lemma}

\begin{proof}
Recall that $\lambda(x,y)=x^2-y-(y-2)(y+2-x^2)S_k(y)S_{k-1}(y)$. As in Lemma \ref{alpha}, we have that $S_k(y)$, $S_{k-1}(y)$ and $(y+2-x^2)$ are positive when $y\geq 2$. Hence, $\lambda(x,y)-x^2+y=-(y-2)(y+2-x^2)S_k(y)S_{k-1}(y)\leq 0$ for all $y\geq 2$.

Now $\lambda(x,y)\leq x^2-y, $ so letting $y\to \infty$ we see that $\lambda(x,y)$ tends to $-\infty$ as $y$ grows. We also have that $\lambda(x,2)=x^2-2$. Since $\lambda$ is a continuous function, the lemma follows.
\end{proof}

\begin{theorem}\label{Polybigtheorem}
Fix $n\geq 2$, and let $K=J(2k+1,2m)$. Then there is a presentation of $\pi_1(X_K)$ with Riley polynomial $\phi_K(x,y)$ such that $\phi_K(2\cos(\pi/n),y)$ has a root $y_n>2$ in the following cases:

\begin{enumerate}
\item $n\geq 3$ when $m\leq -3$
\item $n\geq 4$ when $m=-2$
\item $n\geq 5$ when $m=2$
\item $n\geq 4$ when $m=3$
\item $n\geq 3$ when $m\geq 4$.
\end{enumerate}
\end{theorem}

\begin{proof}Again we choose the presentation of $\pi_1(X_K)$ as the one given in (\ref{favoritepresentation}). We will argue carefully the case of $m$ positive; the case of $m$ negative is argued similarly. When $m=2$ the result is proved by Proposition \ref{m=2}. We now assume that $m\geq 3.$ In this case we have that $\phi_K(x_n,y)=S_{m-1}(\lambda)\alpha-S_{m-2}(\lambda)$.

We proceed by noting that by Lemma \ref{leading}, there is  $y_0\geq 2$ such that $(-1)^{m}\phi_K(x_n,y_0)< 0$. Finding another $y_1\geq 2$, with $(-1)^{m}\phi_K(x_n,y_1)> 0$, would give us a root larger than $2$ by the intermediate value theorem. 

Let $(r_1,r_2)=\left(2\cos\left(\frac{(m-1)\pi}{m}\right)\!, 2\cos\left(\frac{(m-2)\pi}{m}\right)\right)$.  Lemma \ref{Sbehav} gives that $(-1)^{m-1}S_{m-1}(t)<0$ on $(r_1,r_2)$. We also have that $c=2\cos\left(\frac{(m-2)\pi}{m-1}\right)$ is a root of $S_{m-2}(t)$. Note that $c\in (r_1,r_2)$ for $m\geq 3$.  By Lemma \ref{anylambda}, if we assume $c\leq x_n^2-2$ then there is $y_c\geq 2$ so that $\lambda(x_n,y_c)=c$. Combined with the fact that $\alpha>1$ by Lemma \ref{alpha}, we have that
\begin{align*}
(-1)^{m-1}\phi_K(x_n,y_c)&=(-1)^{m-1}(S_{m-1}(\lambda(x_n,y_c))\alpha(x_n,y_c)-S_{m-2}(\lambda(x_n,y_c)))\\
&=(-1)^{m-1}(S_{m-1}(c)\alpha-S_{m-2}(c))\\
&=(-1)^{m-1}S_{m-1}(c)\alpha< (-1)^{m-1}S_{m-1}(c)< 0.
\end{align*}

\noindent Thus, as long as $c\leq x_n^2-2$ we have that $y_c=y_1\geq 2$ is the value we seek. 

Before computing when $c\leq x_n^2-2$ holds, we highlight how the case of $m$ negative differs. The argument is similar; the differences come from the fact that in this case $\phi_K(x_n,y)=S_{|m|}(\lambda)-S_{|m|-1}(\lambda)\alpha$. Now, let $(r_1',r_2')=\left(2\cos\left(\frac{|m|\pi}{|m|+1}\right), 2\cos\left(\frac{(|m|-1)\pi}{|m|+1}\right)\right)$ and $c'=2\cos\left(\frac{(|m|-1)\pi}{|m|}\right)$ and the we obtain a root of $\phi_K(x_n,y)$ so long as $c'\leq x_n^2-2$.

To conclude we need only determine when $c\leq x_n^2-2$ and when $c'\leq x_n^2-2$. Recall that $x_n=2\cos(\pi/n)$. For $m=3$ we see that $c=0\leq x_n^2-2=4\cos(\pi/n)^2-2$ so long as $n\geq 4$. Similarly if $m\geq 4$ then $c\leq-1\leq 4\cos(\pi/n)^2-2$ for all $n\geq 3$. If $m=-2$ then $c'=0\leq x_n^2-2$ so long as $n\geq 4$. If $m\leq -3$ then $c'\leq -1\leq x_n^2-2$ for $n\geq 3$. For these $n$, we can conclude that $\phi_K(x,y)$ has a root $y_n\in (2,\infty)$.
\end{proof}

\section{Another family of two-bridge knots}\label{moreresults}

Let $K_l$ denote the knot in Figure \ref{notDT} for the remaining section. We assume that $l\geq 2$; in this case, it can be shown that $K_l$ is \textit{not} a double-twist knot, and $g(K_l)=l$. These knots are two-bridge and have associated fraction $(10(l-1)+7)/(4(l-1)+3)$.

As for Theorem \ref{Bigtheorem}, to prove left-orderability of $\Sigma_n(K_l)$ for some $n$, we find certain roots of a Riley polynomial of $K$. By Theorem \ref{itsallaboutRiley}, the following theorem implies Theorem \ref{bigtwistythm}.

\begin{theorem}\label{bigtroots}
There is a presentation of $\pi_1(X_{K_l})$ with Riley polynomial $\phi_{K}(x,y)$ which has a root $y_n>2$ in the following cases:

\begin{enumerate}
\item $n\geq 5$ when $l=2$
\item $n\geq 4$ when $l=3$
\item $n\geq 3$ when $l\geq 4$
\end{enumerate}

\end{theorem}
\subsection{Fundamental groups of two-bridge knots}

A two-bridge knot group has a presentation determined by a sequence of signs $S(p,q)$ which we now describe. Let $n\in\mathbb{Z}$ with $(n,p)=1$; then let  $\overline{n}$ denote the choice of representative of $n$ modulo $2p$ with $-p<\overline{n}<p$. Then let $S(p,q)=\{\epsilon_1,\ldots, \epsilon_{p-1}\}$ denote the (ordered) set of signs of the representatives $\overline{iq}$ for $1\leq i<p$. In other words, $\epsilon_i=|\overline{iq}|/\overline{iq}$.

\begin{proposition}[Proposition 1 in \cite{Ri72}]\label{epsilons} Let $K$ be the two-bridge knot with fraction $(p,q)$ with $q$ odd. Then $\pi_1(X_K)$ has a presentation
\[\pi_1(X_K)= \langle a,b|va=bv\rangle\]

\noindent where $v=a^{\epsilon_1}b^{\epsilon_2}\cdots a^{\epsilon_{p-2}}b^{\epsilon_{p-1}}$.
\end{proposition}

For convenience of notation, a sequence of signs $S(p,q)$ can be abbreviated so that consecutive instances of $+1$ or $-1$ in the sequence $k$ times will be denoted $\langle k\rangle$ and $\langle -k\rangle$ respectively. For example, the sequence $\{1,1,-1,-1\}$ will be abbreviated $\langle 2\rangle \langle -2\rangle$. A sequence can also be shortened by denoting a repeated subsequence using exponents; for example we might write $(\langle 3\rangle \langle -2\rangle)^2=\langle 3\rangle \langle -2\rangle\langle 3\rangle \langle -2\rangle$.

Using this notation, a sequence of signs always has the form $S(p,q)=\langle c_1\rangle \langle -c_2\rangle \cdots \langle -c_{k-1}\rangle\langle c_k\rangle$ where $c_i>0$. The only non-trivial part of this statement is that the sequence always ends on a $+1$. It is not difficult to check that $\epsilon_{p-1}=+1$ always holds.

We now discuss the reduction operation of Hirasawa-Murasugi which can be performed to a sequence of signs $S(p,q)$ \cite{HM07}. Let $p=mq+r$ where $m\geq 2$ and $0<r<p$. The reduction operation takes $S=\langle c_1\rangle \langle -c_2\rangle \cdots \langle -c_{k-1}\rangle\langle c_k\rangle$ and yields first $S_1^*=\{\langle c_1^*\rangle \langle -c_2^*\rangle\cdots \langle c_k^*\rangle\}$ where $c_i^*=c_i-2$. If it happens that $c_i^*=0$, then either we have 
\[\begin{array}{lcr}
\langle c_{i-1}^*\rangle \langle 0\rangle \langle c_{i+1}^*\rangle  &  \mbox{or} & \langle -c_{i-1}^*\rangle \langle 0\rangle \langle -c_{i+1}^*\rangle
\end{array}\]

\noindent in the sequence $S_1^*$. In this case it makes sense to combine the same-sign terms in the first case simply to $\langle c_{i-1}^*+c_{i+1}^*\rangle$ and the second to $\langle -( c_{i-1}^*+c_{i+1}^*)\rangle$. After removing zeros $\langle 0\rangle$ in the sequence $S_1^*$ and combining the same-sign terms we obtain a simpler sequence denoted $S^*(p,q)$.

\begin{proposition}[Proposition 7.1 in \cite{HM07}] \label{reduce}
Let $S(p,q)$ denote the sequence of signs for the fraction $(p,q)$, and write $p=mq+r$ with $m\geq 2$ and $0<r<p$. Then the reduced sequence $S^*(p,q)=S(p-2q,q)$.
\end{proposition}
\begin{lemma}[Proposition 6.1 in \cite{HM07}]\label{shapeofS}
Let $S(p,q)=\langle c_1\rangle \langle -c_2\rangle \cdots \langle -c_{k-1}\rangle\langle c_k\rangle$ be a sequence of signs for the two-bridge knot $K(p,q)$ with $p=mq+r$. Then $c_i$ is either $m$ or $m+1$, and $c_1=c_k=m$.
\end{lemma}

\begin{proposition} The sequences of signs for the two-bridge knots $K(10s+7,4s+3)$  is $S(10s+7,4s+3)=\langle 2\rangle\langle -2\rangle(\langle 3\rangle\langle -2\rangle)^{2s}\langle 2\rangle=(\langle 2\rangle\langle -2\rangle\langle 3\rangle\langle -2\rangle\langle 1\rangle)^s\langle 2\rangle\langle -2\rangle\langle 2\rangle$ for any integer $s\geq 1$.
\end{proposition}

\begin{proof}
We first compute the reduced sequence $S^*(10s+7,4s+3)=S(2s+1,4s+3)$ by Proposition \ref{reduce}. Now $S(2s+1,4s+3)$ is the sequence of signs for the knot $K(2s+1,4s+3)\sim K(2s+1,1)$ since $4s+3\equiv 1$ (mod $2s+1)$. Thus we have that $S^*(10s+7,4s+3)=S(2s+1,1)$. It is simple to compute $S(2s+1,1)=\langle 2s\rangle$.

We work backwards now to write $S(10s+7,4s+3)$. Since \sloppy $S(10s+7,4s+3)=\langle c_1\rangle \langle -c_2\rangle \cdots \langle -c_{k-1}\rangle\langle c_k\rangle$ with each $c_i=2$ or $3$, we have that each $c_i^*=0$ or $1$. After reducing we can conclude that $S(2s+1,1)=S^*(10s+7,4s+3)=\langle 1\rangle\langle -1\rangle\cdots \langle 1\rangle$. Since we know that in fact $S(2s+1,1)=\langle 2s\rangle$ we can conclude that no $c_{2i}=3$ for any $i$, so $c_{2i}=2$.

Now we know the sequence is of the following form.
\[\langle 2\rangle\langle -2\rangle\langle c_{i_1}\rangle\langle -2\rangle\langle c_{i_2}\rangle\cdots\langle -2\rangle\langle c_{i_{t}}\rangle\langle -2\rangle\langle 2\rangle\]

\noindent where $c_{i_1},\ldots c_{i_{t}}$ are equal to either $2$ or $3$.

We note that the sum $c_1+c_2+\ldots c_k=p-1$ for any $S(p,q)$. Exactly $2s$ of these unknown $c_i$ must equal $3$, since only these will reduce to $\langle 1\rangle$ in $S(2s+1,1)$. Assuming for the moment that $t=2s$ and all $c_{i_1},\ldots c_{i_t}=3$ gives that
\[\sum_{i=1}^k c_i\geq 3(2s)+(2s-1)2+8=10s+6=p-1\]

\noindent which achieves the maximum value. Thus, it must in fact be the case that 
\[S(p,q)=\langle 2\rangle\langle -2\rangle(\langle 3\rangle\langle -2\rangle)^{2s}\langle 2\rangle.\qedhere\]
\end{proof}

\noindent Proposition \ref{epsilons} now gives that the knot group for $K_l=K(10(l-1)+7,4(l-1)+3)$ is
\[\pi_1(X_{K_l})=\langle a,b|va=bv\rangle\]

\noindent where $v=c^{(l-1)}d$ with $c=aba^{-1}b^{-1}abab^{-1}a^{-1}b$ and $d=aba^{-1}b^{-1}ab$.

\subsection{Towards a proof of Theorem \ref{bigtroots}}
Let $\rho$ be the map defined in (\ref{goodrep}). Then we will again call $\phi_K(x,y)=f(s+s^{-1},y)$ the Riley polynomial of $K$ (associated to the given presentation) where $R=VA-BV$.

\begin{lemma} Let $K=K_l$. Then $\phi_K(x,y)= S_{l-1}(\lambda)\alpha-S_{l-2}(\lambda)\beta$, where $\alpha$, $\beta$ and $\lambda$ are polynomials in $n$ and $y$.
\end{lemma}

\begin{proof}
Let $\rho(c)=C$ and $\rho(d)=D$. Using Lemma \ref{matrixpower} we have that $C^{l-1}=S_{l-1}(\trace(C))I-S_{l-2}(\trace(C))C^{-1}$ so that
\begin{align*}
VA-BV=C^{l-1}DA-BC^{l-1}D&=S_{l-1}(\trace(C))(DA-BD)-S_{l-2}(\trace(C))(C^{-1}DA-BC^{-1}D)
\end{align*}

Direct computation gives that $\lambda=\lambda(x,y):=\trace(C)=9 x^2 - 12 x^4 + 4 x^6 - 5 y + 10 x^2 y + 2 x^4 y - 4 x^6 y - 
 11 x^2 y^2 + 8 x^4 y^2 + x^6 y^2 + 5 y^3 - 4 x^2 y^3 - 3 x^4 y^3 + 
 3 x^2 y^4 - y^5.$
We can also compute directly that
\begin{align*}
DA-BD=\begin{bmatrix}
0& \alpha(x,y)\\
(y-2)\alpha(x,y)&0
\end{bmatrix}
\end{align*}

where $\alpha=\alpha(x,y)=1 - 4 x^2 + 2 x^4 + 2 y - x^2 y - x^4 y - y^2 + 2 x^2 y^2 - y^3$, and
\begin{align*}
C^{-1}DA-BC^{-1}D=\begin{bmatrix}
0& \beta(x,y)\\
(y-2)\beta(x,y)&0
\end{bmatrix}
\end{align*}

where $\beta=\beta(x,y)=-1+x^2-y$. Thus

\[
VA-BV=S_{l-1}(\lambda)\begin{bmatrix}
0& \alpha\\
 (y-2)\alpha &0
\end{bmatrix}-S_{l-2}(\lambda)\begin{bmatrix}
0& \beta\\
 (y-2)\beta &0
\end{bmatrix}
\qedhere
\]
\end{proof}

\noindent Directly computing the leading term of the polynomial $\phi_K(x,y)$ as in Lemma \ref{leading} we can obtain the following.

\begin{lemma}\label{endbehav}
$\lim_{y\to \infty} (-1)^l\phi_K(x,y)=\infty$.
\end{lemma}

\begin{lemma}\label{alphadecr}  The function $\alpha(x,y)$ is strictly decreasing as a function of $y$ on $[2,\infty)$ for all $x=x_n\in [1,2]$, or in other words, for all $n\geq 3$.
\end{lemma}

\begin{proof}
We can compute $\frac{d\alpha}{dy}=\alpha'(x,y)=2-x^2-x^4-2y+4x^2y-3y^2.$ For fixed $x$ the end behavior of $\alpha'(x,y)$ is clearly decreasing. To conclude on which interval the function $\alpha'$ is negative we note that the discriminant of this polynomial is 
\[
(-2+4x^2)^2-4(-3)(2-x^2-x^4)=4x^4-28x^2+28
\]

Now assume that $x=x_n=2\cos(\pi/n)$. This function is negative on the interval $[\sqrt{2},2)$, that is for all $x_n$ with $n\geq 4$. Thus $\alpha'$ has no real roots, and is negative for all $y\in \mathbb{R}$. The case $n=3$ can be checked by hand.
\end{proof}

\noindent Direct computation gives the following lemma.

\begin{lemma}\label{ablfacts} We have the following facts for the polynomials $\alpha$ and $\lambda$.
\begin{enumerate}
\item $\alpha(x,x^2-1)=-1$
\item $\lambda(x,2)=x^2-2$
\item $\lambda(x,x^2-1)=1$.
\end{enumerate}
\end{lemma}

\begin{proposition}\label{corabl} Suppose both $c\leq 1$ and $c\leq x_n^2-2$ hold. Then there is $y_c$ satisfying both $y_c\geq x^2_n-1$ and $y_c\geq 2$ such that $\lambda(x_n,y_c)=c$ and $\alpha(x_n,y_c)<0$.
\end{proposition}

\begin{proof}
The leading term of $\lambda(x,y)$ gives that $\lim_{y\to \infty}\lambda(x,y)=-\infty$. Thus, by Lemma \ref{ablfacts}(b) and (c) if $c\leq x_n^2-2$ and $c\leq 1$ both hold, then the claim holds by the intermediate value theorem. 

By Lemmas \ref{alphadecr} and \ref{ablfacts}(a) and the intermediate value theorem, we have that $\alpha(x_n,y_c)<0$ for any such $y_c$. 
\end{proof}

\begin{proposition}\label{m=1} Let $l=2$. Then $\phi_K(x,y)$ has a root $y_n>2$ for all $n\geq 5$.
\end{proposition}
\begin{proof} For $l=2$ we have that $\phi_K(x,y)=\lambda(x,y)\alpha(x,y)-\beta(x,y)$. By Lemma \ref{ablfacts} we can evaluate
\begin{align*}
\phi_K(x,x_n^2-1)&=\lambda(x,x^2-1)\alpha(x,x_n^2-1)-\beta(n,x_n^2-1)\\
&=(1)(-1)-0=-1
\end{align*}

Thus, if $x_n^2-1\geq 2$, which will be the case for $n\geq 6$, we have that $y_-=x_n^2-1\geq 2$ is a value on which $\phi_K(x,y)$ is negative. By Lemma \ref{endbehav}, we know that there exists a $y_+<0$ for which $\phi_K(x,y_+)>0$. Thus, by the intermediate value theorem, there is a root $y_n>2$ for $\phi_K(x,y)$ whenever $n\geq 6$.

When $n=5$, it can be checked by hand that $\phi_K(5,y)$ also has a root $y_5>2$.
\end{proof}

\begin{proof}(Theorem \ref{bigtroots}).
Recall that we seek a $y_n>2$ such that $\phi_K(x_n,y_n)=0$. Fix $n\geq 2$.

In light of Proposition \ref{m=1}, we can assume $l\geq 3$.  By Lemma \ref{endbehav} we have that there is $y_0>2$ such that $(-1)^{l-1}\phi_K(x_n,y_0)<0$. Now we will find $y_1\geq 2$ such that $(-1)^{l-1}\phi_K(x_n,y_1)>0$.

Let $c=2\cos\left(\frac{(l-2)\pi}{l-1}\right)$. Note that $(-1)^{l-1}S_{l-1}(c)<0$ on $(r_1,r_2)=\left(2\cos\left(\frac{(l-1)\pi}{l}\right)\!,\,2\cos\left(\frac{(l-2)\pi}{l}\right)\right)$. It can be checked that for $l\geq 3$ we have $c\in I$.

Note that $c\leq 1$ always holds when $l\geq 3$. If, in addition, we have that $c\leq x_n^2-2$, then by Lemma \ref{ablfacts}, there is $y
_c\geq 2$ such that $\lambda(x_n,y_c)=c$ and $\alpha(x_n,  y_c)<0$. Hence,
\begin{align*}
(-1)^{l-1}\phi_K(x_n,y_c)&=(-1)^l(S_{l-1}(\lambda(x_n,y_c))\alpha(x_n,y_c)-S_{l-2}(\lambda(x_n,y_c))\beta(x_n,y_c))\\
&=(-1)^{l-1}(S_{l-1}(c)\alpha(x_n ,y_c)-S_{l-2}(c)\beta(x_n,y_c))\\
&=(-1)^{l-1}S_{l-1}(c)\alpha(x_n,y_c)>0.
\end{align*}

Thus, $y_c=y_1$ is the value we seek. To finish, we need to conclude when $c\leq x_n^2-2$ holds. This occurs for $n\geq 4$ if $l=3$ and for $n\geq 3$ if $l\geq 4$.
\end{proof}

\bibliography{LOdoubletwist}

\end{document}